\documentclass[letterpaper,12pt]{amsart}
\usepackage{latexsym,array,delarray,amsthm,amssymb,epsfig,color}%eufrak
\usepackage{graphics,graphicx, caption, comment}

\theoremstyle{plain}
\newtheorem{thm}{Theorem}[section]
\newtheorem{lemma}[thm]{Lemma}
\newtheorem{prop}[thm]{Proposition}
\newtheorem{cor}[thm]{Corollary}
\newtheorem{conj}[thm]{Conjecture}
\newtheorem*{thm*}{Theorem \ref{thm:main}}
\newtheorem*{lemma*}{Lemma}
\newtheorem*{prop*}{Proposition}
\newtheorem*{cor*}{Corollary}
\newtheorem*{conj*}{Conjecture}

\theoremstyle{definition}
\newtheorem{defn}[thm]{Definition}
\newtheorem{ex}[thm]{Example}
\newtheorem{pr}[thm]{Problem}

\theoremstyle{remark}
\newtheorem*{rmk}{Remark}

\numberwithin{figure}{section}

\newcommand{\qq}{\mathbb{Q}}
\newcommand{\rr}{\mathbb{R}}
\newcommand{\cc}{\mathbb{C}}
\newcommand{\kk}{\mathbb{K}}

\newcommand{\bfp}{\mathbf{p}}
\newcommand{\bfq}{\mathbf{q}}

\newcommand{\ca}{\mathcal{A}}

\newcommand{\cali}{\mathcal{I}}

\newcommand{\ind}{\mbox{$\perp \kern-5.5pt \perp$}}

\newcommand{\mlt}{\mathrm{mlt}}
\newcommand{\wmlt}{\mathrm{wmlt}}
\newcommand{\rank}{\mathrm{rank}}
\newcommand{\birank}{\mathrm{birank}}
\newcommand{\diag}{\mathrm{diag}}
\newcommand{\smt}{\mathrm{smt}}

\begin{document}

\title[Maximum likelihood threshold]{The maximum likelihood threshold of a graph }
\author{Elizabeth Gross}
\author{Seth Sullivant}

%\institute{F. Author \at
             % first address \\
              %Tel.: +123-45-678910\\
             % Fax: +123-45-678910\\
             \email{elizabeth.gross@sjsu.edu}
              \email{smsulli2@ncsu.edu } 
                          \address{Department of Mathematics and Statistics, One Washington Square,  San Jos\'{e} State University, San Jos\'{e}, CA, 95192-0103, USA}
              \address{Department of Mathematics, Box 8205, North Carolina State University, Raleigh, NC, 27695-8205, USA }     %  \\
   
%             emph{Present address:} of F. Author  %  if needed
          % \and
          % S. Author \at
           
%}

\maketitle

\begin{abstract} The maximum likelihood threshold of a graph is the smallest number of data points that guarantees that maximum likelihood estimates exist almost surely in the Gaussian graphical model associated to the graph. We show that this graph parameter is connected to the theory of combinatorial rigidity. In particular, if the edge set of a graph $G$ is an independent set in the $n-1$-dimensional generic rigidity matroid, then the maximum likelihood threshold of $G$ is less than or equal to $n$. This connection allows us to prove many results about the maximum likelihood threshold.  We conclude by showing that these methods give exact bounds on the number of observations needed for the score matching estimator to exist with probability one.
\end{abstract}

%%%%%%%%%%%%%%%%%%%%%%%%%%%%%%%%%%%%%%%%%%%%%%%%
%%%%%%%%%%%%%%%%%%%%%%%%%%%%%%%%%%%%%%%%%%%%%%%%
%%%%%%%%%%%%%%%%%%%%%%%%%%%%%%%%%%%%%%%%%%%%%%%%
%%%%%%%%%%%%%%%%%%%%%%%%%%%%%%%%%%%%%%%%%%%%%%%%

\section{Introduction}
Let $X=(X_1, \ldots, X_m)$ be a $m$-dimensional random vector distributed according to a multivariate normal distribution, i.e. $X \sim \mathcal N(\mu, \Sigma)$.  In a Gaussian graphical model, an undirected graph $G=(\{1, \ldots, m\}, E)$ encodes the conditional independence structure of the distribution: the edge $(i,j) \notin E$ if and only if $X_i$ and $X_j$ are conditionally independent given the remaining variables.  Originally introduced by Dempster \cite{Dempster1972} under the name of covariance selection models, Gaussian graphical models have found a variety of applications, especially in systems biology and bioinformatics.   For example, these models are used to model gene regulatory networks \cite{Dobra2004, SS2005} and to infer pathways in metabolic networks \cite{KSIAT11}. Lauritzen \cite{Lauritzen1996} and Whittaker \cite{Whittaker1990} both give general introductions to Gaussian graphical models.

In this paper, we are concerned with the existence of the maximum likelihood estimator (MLE) of the covariance matrix when the mean vector $\mu=0$.  %Gaussian graphical models are regular exponential families, and thus, it is well known that the MLE exists if and only if the sufficient statistic lies in the interior of its convex support 
For Gaussian graphical models, when the number of observations $n$ is larger than the number of random variables $m$, the MLE exists with probability one.  But it is often the case, especially in biological applications, that $m \gg n$. In this setting, it is still possible that the MLE exists with probability one, which invites the question: \emph{For a given graph $G$, what is the smallest $n$ such that the maximum likelihood estimator of $\Sigma$ exists with probability one?}  We denote the resulting graph invariant by $\mlt(G)$ and call it the \emph{maximum likelihood
threshold}.

As originally proven in \cite{Dempster1972} and discussed further in \cite{Uhler2012}, the existence of the MLE for given data set and
for a particular Gaussian graphical model is equivalent to the
existence of a 
full rank matrix completion of the incomplete matrix obtained by keeping only the diagonal entries and entries corresponding to $E$ of the sample covariance matrix $\Sigma_0$. Let $\mathbb{S}^{m}$ denote the set of $m \times m$
symmetric matrices, $\mathbb{S}^{m}_{>0}$ the set of $m \times m$ positive definite
symmetric matrices, and $\mathbb{S}^{m}_{\ge 0}$ the set of $m \times m$ positive
semidefinite symmetric matrices.   Let $Sym(m,n)$ denote the set of $m \times m$
symmetric matrices of rank $\leq n$.  Let 
\begin{equation}\label{eq:phiG}
\phi_{G} :  \mathbb{S}^{m}  \rightarrow \rr^{V + E},  \quad \quad
\phi_{G}( \Sigma)  =  (\sigma_{ii})_{i \in V}  \oplus (\sigma_{ij})_{ij \in E} 
\end{equation}
be the coordinate projection that extracts the diagonal and entries corresponding
to edges of $G$ of 
the symmetric matrix $\Sigma = (\sigma_{ij})_{i,j \in V}$.  In the setting of matrix completion problems,  the question of determining the maximum likelihood threshold of a graph $G$ is:

\begin{pr}[Maximum Likelihood Threshold]\label{pr:main}
Given a graph $G$, what is the smallest $n$ such that for almost all
$\Sigma_{0}  \in Sym(m,n) \cap \mathbb{S}^{m}_{\ge 0}$
there exists a $\Sigma \in \mathbb{S}^{m}_{>0}$ such that
$\phi_{G}(\Sigma_{0})  =  \phi_{G}(\Sigma)$?  
\end{pr}

%Now every positive semi-definite matrix of rank $n$ arises as 
%$P^{T}P$ for some $n \times m$ real matrix, with real columns
%$\bfp_{i}  \in \rr^{n}$.  Problem \ref{pr:main} is equivalent to asking: \emph{Given a graph $G$, what is the smallest $n$ such that for almost
%all $P =  ( \bfp_{1}, \ldots, \bfp_{m}) \in \rr^{n \times m}$
% there exists a set of linearly independent vectors
%$P'  =  (\bfq_{1}, \ldots,  \bfq_{m}) \in \rr^{m \times m}$
%such that} 
%$$
%\| \bfp_{i} \|_{2}  =  \| \bfq_{i} \|_{2} \mbox{ for all } i, \mbox {and }
%\bfp_{i} \cdot \bfp_{j}  = \bfq_{i} \cdot \bfq_{j}  \mbox{ for all } ij \in E?
%$$
%This framework results in a natural connection between the \emph{symmetric minor matroid}  and the \emph{generic rigidity matroid}, which we will use to bound the maximum likelihood threshold in our main result, Theorem \ref{thm:equiv}.

Since every positive semi-definite matrix of rank $n$ arises as 
$P^{T}P$ for some $n \times m$ real matrix $P$ with real columns
$\bfp_{i}  \in \rr^{n}$,   problem \ref{pr:main} is equivalent to asking: \emph{Given a graph $G$, what is the smallest $n$ such that for almost
all $P =  ( \bfp_{1}, \ldots, \bfp_{m}) \in \rr^{n \times m}$
 there exists a set of linearly independent vectors
$Q  =  (\bfq_{1}, \ldots,  \bfq_{m}) \in \rr^{m \times m}$
such that} 
$$
\| \bfp_{i} \|_{2}  =  \| \bfq_{i} \|_{2} \mbox{ for all } i, \mbox {and }
\bfp_{i} \cdot \bfp_{j}  = \bfq_{i} \cdot \bfq_{j}  \mbox{ for all } ij \in E?
$$
This formulation results in a natural connection between the \emph{symmetric minor matroid}  and the \emph{generic rigidity matroid}, which we will use to bound the maximum likelihood threshold in our main result:

\begin{thm}\label{thm:mainintro}
If the edge set of a graph $G$ is an independent set in the $n-1$-dimensional generic rigidity matroid, then the maximum likelihood threshold of $G$ is less than or equal to $n$.
\end{thm}

\medskip

In spite of the
seeming importance of the maximum likelihood threshold in applications where $n \ll m$,  very little
is known about the value of $\mlt(G)$ except in certain special instances.  Some of these instances are straightforward:
\begin{itemize}
\item  $\mlt(G) = 1$ if and only if $G$ has no edges,
\item  $\mlt(G) = 2$ if and only if $G$ has no cycles, and
\item  $\mlt(G) = m = \#V$ if and only if $G  = K_{m}$.
\end{itemize}
For more complicated graphs, Buhl showed in \cite{Buhl1993} that the $\mlt(G)$ is bounded in terms of the clique number $\omega(G)$ and treewidth $\tau(G)$ of the graph.  Recall that the clique number of a graph is the number of vertices
of the largest complete subgraph of $G$.  The treewidth of a graph is one 
less than the clique number of the smallest chordal cover of $G$.  

\begin{prop}\cite[Cor 3.3]{Buhl1993}\label{prop:Buhl}
Let $G$ be a graph.  Then
$$
\omega(G) \leq  \mlt(G)  \leq  \tau(G) + 1.
$$
\end{prop}

Proposition \ref{prop:Buhl} implies that if $G$
is chordal then $\mlt(G) = \omega(G)$.  
However, in general, these bounds
are far from optimal and far from one another.
For instance, there are graphs with $\omega(G) = 2$ and arbitrarily large treewidth.  

In this paper, we develop the connection between the maximum likelihood threshold and combinatorial rigidity theory through the \emph{rank}\footnote{ A related but non-equivalent problem to Problem \ref{pr:main} was recently explored by Ben-David \cite{BenDavid2014} and asks ``For each graph $G$, what is the smallest $n$ such that \emph{for every} 
$\Sigma_{0}  \in Sym(m,n) \cap \mathbb{S}^{m}_{\ge 0}$
\emph{in general position} there exists a $\Sigma \in \mathbb{S}^{m}_{>0}$ such that
$\phi_{G}(\Sigma_{0})  =  \phi_{G}(\Sigma)$?"  Such an $n$ is an upper bound on the $\mlt(G)$ and is referred to as the \emph{Gaussian rank} of $G$ in \cite{BenDavid2014}.  The Gaussian rank of a graph is different from the \emph{rank} of a graph that we explore in this paper.} of a graph.

%More recent progress was made by Uhler in \cite{Uhler2012}, in which Uhler bounds the maximum likelihood threshold using the \emph{rank} \footnote{ A related but non-equivalent problem to Problem \ref{pr:main} was recently explored by Ben-David \cite{BenDavid2014} and asks ``For each graph $G$, what is the smallest $n$ such that \emph{for every} 
%$\Sigma_{0}  \in Sym(m,n) \cap \mathbb{S}^{m}_{\ge 0}$
%\emph{in general position} there exists a $\Sigma \in \mathbb{S}^{m}_{>0}$ such that
%$\phi_{G}(\Sigma_{0})  =  \phi_{G}(\Sigma)$?"  Such an $n$ is an upper bound on the $\mlt(G)$ and is referred to as the \emph{Gaussian rank} of $G$ in \cite{BenDavid2014}.  The Gaussian rank of a graph is different from the \emph{rank} of a graph that we explore in this paper.} of a graph.

\begin{defn}
The \emph{rank}   of a graph $G$, denoted ${\rm rank}(G)$, is the smallest
$n$ such that $ \dim \phi_{G} (Sym(m,n))  =  \#V + \#E$.  
\end{defn}
\noindent In \cite{Uhler2012}, Uhler  showed the following bound relating the maximum likelihood threshold and
the rank of $G$.

\begin{thm}\cite[Thm 3.3]{Uhler2012}\label{thm:uhler}
Let $G$ be a graph.  Then 
$$
\mlt(G)  \leq {\rm rank}(G).
$$
\end{thm}

It is still unknown whether there exists a graph such 
that $\mlt(G) < {\rm rank}(G)$, but
this might be because the maximum likelihood threshold is so poorly understood. 
The main
goal of this paper is to develop a better understanding of the
notion of the rank of a graph, so that we can develop better bounds on 
the maximum likelihood threshold.   
One always has ${\rm rank}(G) \leq \tau(G) + 1$, but usually ${\rm rank}(G)$ is
significantly smaller than $\tau(G) + 1$, which yields substantially
improved bounds.   For example, for an arbitrary $k_{1} \times k_{2}$ 
grid with $k_{1}, k_{2} \geq 2$, denoted  $Gr_{k_{1},k_{2}}$, 
$\rank(Gr_{k_{1},k_{2}}) = \mlt(Gr_{k_{1}, k_{2}}) = 3$, whereas 
$\tau(Gr_{k_{1},k_{2}})+1 = \min(k_{1},k_{2})+1$ is substantially larger 
(Corollary \ref{cor:grid}).

\medskip

While the question of whether or not there is a gap between $\mlt(G)$ and $\rank(G)$ remains open, we conclude the paper by turning our attention to another estimator, the \emph{score matching estimator} (SME). The \emph{score matching threshold}
is the smallest amount of data such that the SME exists with probability one.
Theorem \ref{thm:smt} states that the score matching threshold of a graph $G$ is equal to its rank.  
The SME was introduced in \cite{Hyvarinen2005} and furthered studied in \cite{Forbes2015}.  The score matching equations are linear, so when the SME exists, computing the estimator is efficient even for large dense graphs.  Hence, the SME has promising applications to model selection for high dimensional graphical models.

As Lauritzen and Forbes point out in \cite{Forbes2015}, a simple sufficient condition for the existence of the SME would be advantageous, since it could be used to limit model searches.  By the same reasoning, simple sufficient conditions on the existence of the MLE are desirable as well. Corollary \ref{cor:lamansthm} of Section 3 and Corollary \ref{cor:smt3} of Section 6, give such sufficient conditions for the $\mlt(G)$ and $\smt(G)$ when $n=3$; the conditions can be checked in $O(\#V(G) \cdot \#E(G))$ time \cite{Jacobs1997}.

%A key observation is that ${\rm rank}(G)$ is closely related to combinatorial
%rigidity theory.  Results in that
%well-established theory can be used to produce new bounds on $\mlt(G)$.
%As we will see in the next section, ${\rm rank}(G)$ is the smallest $n$ for which
%the set of edges of $E(G)$ are independent in the 
%generic rigidity matroid $\mathcal{A}(n-1)$.  
The remainder of the paper is organized as follows.  In Section \ref{sec:crt}, we introduce algebraic matroids, in particular, the symmetric minor matroid and the combinatorial rigidity matroid.  Within this matroidal setting, we show that the ${\rm rank}(G)$ is the smallest $n$ for which the set of edges of $E(G)$ are independent in the  generic rigidity matroid $\mathcal{A}(n-1)$.    In Section \ref{sec:basic}, we provide a 
brief summary of the consequences of this connection for the maximum likelihood threshold.
In Section \ref{sec:splitting}, we provide a splitting theorem which allows for the 
computation of improved bounds on $\rank(G)$ by reducing to smaller graphs,
at the expense of calculating the \emph{birank} of bipartite graphs.
In Section \ref{sec:weak}, we introduce the notion of weak maximum likelihood threshold,
and we provide a splitting lemma and bounds for the weak maximum likelihood threshold based
on the chromatic number.  Finally, in Section \ref{sec:smt}, we show that the $\rank(G)$ is equal to $\smt(G)$ and discuss consequences.

%%%%%%%%%%%%%%%%%%%%%%%%%%%%%%%%%%%%%%%%%%%%%%%%
%%%%%%%%%%%%%%%%%%%%%%%%%%%%%%%%%%%%%%%%%%%%%%%%
%%%%%%%%%%%%%%%%%%%%%%%%%%%%%%%%%%%%%%%%%%%%%%%%
%%%%%%%%%%%%%%%%%%%%%%%%%%%%%%%%%%%%%%%%%%%%%%%%

\section{Combinatorial Rigidity Theory}\label{sec:crt}

In this section, we relate the rank of a graph to  combinatorial 
rigidity theory.  This connection is explained via certain algebraic matroids, 
which we define here.  See \cite{Oxley2011, Welsh1976} for more background 
on matroids and \cite{Graver1993, Whiteley1996} for background on rigidity theory.   Both the  rigidity matroid and symmetric minor matroids are discussed in detail in Section 3 of \cite{KRT2013} in the context of matroids with symmetries.

\begin{defn}
Let $S$ be a set and $\mathcal{I}$ a collection of subsets of $S$ satisfying:
\begin{enumerate}
\item  $\emptyset \in \cali$
\item  If $X \in \cali$ and $Y \subseteq X$ then $Y \in \cali$, and
\item  If $X, Y \in \cali$ with $|X| < |Y|$ then there is a $y \in Y$
such that $X \cup \{ y \} \in \cali$.
\end{enumerate}
The pair $(S, \cali)$ is called a \emph{matroid} and the elements
of $\cali$ are called \emph{independent sets}.

\end{defn}

The protypical example of a matroid comes from linear algebra:
if $S$ is a set of vectors and $\cali$ consists of all linearly independent
subsets then the pair $(S, \cali)$ is a matroid.  Other terminology
from matroid theory comes from linear algebra.  An independent set
of maximal size in $\cali$ is called a \emph{basis}.  A subset $X \subseteq S$
that contains a basis is said to \emph{span} the matroid.
The main type of matroid that we will need in this work
comes from algebra.

\begin{defn}
Let $\kk$ be a field, and let $S = \{\alpha_{1}, \ldots, \alpha_{d}\}$ be elements
of a field extension $\mathbb{L}/\kk$.  The \emph{algebraic matroid} on $S$
is the matroid whose independent sets are the collections of $X \subset S$ that
are algebraically independent over $\kk$.
\end{defn}

Two typical ways that algebraic matroids arise are via prime ideals and via parametrizations.
In particular, let $I \subseteq \kk[x] :=  \kk[x_{1}, \ldots, x_{n}]$ be a prime ideal,
and consider the field extension 
$K(\kk[x]/I) / \kk$
where $K(\kk[x]/I)$ denotes the field of fractions.  The natural algebraic
matroid to consider in this context is the matroid on the elements $x_{1}, \ldots, x_{n}$.

The algebraic matroid associated to a rational parametrization is described as follows.
Let $\kk(t) :=  \kk(t_{1}, \ldots, t_{e})$ be the field of fractions of $\kk[t] := 
\kk[t_{1}, \ldots, t_{e}]$.  Consider $d$ rational functions $f_{1}, \ldots, f_{d} \in \kk(t)$.
These determine an algebraic matroid in the obvious way.  This is a special case of the
prime ideal description because we can take the presentation ideal $I \subseteq \kk[x]$  of the 
$\kk$-algebra homomorphism $f: \kk[x]  \rightarrow \kk(t),  f(x_{i})  = f_{i}(t)$.
The algebraic matroid on $x_{1}, \ldots, x_{d} \in K(\kk[x]/I)$ is the
same as the algebraic matroid on $f_{1}, \ldots, f_{d}$, precisely because
$I$ is the ideal of relations among $f_{1}, \ldots, f_{d}$.

The generic rigidity matroid $\mathcal{A}(n)$ of dimension $n$ is constructed as follows.
Let $P = (p_{ij})_{i,j \in n, m}$ be an $n \times m$ matrix of algebraically independent 
indeterminates.  Let
$\bfp_{j}$ be the $j$-th column of $P$.  Consider the algebraic matroid on the 
set of ${m \choose 2}$ polynomials 
$$f_{ij}  =  \| \bfp_{i} - \bfp_{j} \|_{2}^{2}\in \rr[p].$$
One should think of this matroid as giving dependence/independence relationships
between the set of distances between $m$ generic points
in $\rr^{n}$.  A graph $G = (V, E)$ with $V=[m]:=\{1, \ldots, m\}$ is called \emph{rigid} if, for generic choices of the
points $\bfp_{1}, \ldots, \bfp_{m} \in \rr^{n}$, the set of distances $f_{ij}$ such that
$ij \in E$, determine all the other distances $f_{ij}$ with $ij  \in {[m] \choose 2}$.
In the matroidal setting that we have introduced here,
we  weaken the condition to allow only finitely many possibilities
for the other missing distances.  In the language of algebraic matroids, this means that
the set of polynomials $\{f_{ij} : ij \in E(G) \}$ is a spanning set for
the algebraic matroid $\mathcal{A}(n)$.
On the other hand, a graph $G$ is \emph{stress-free} precisely when  $\{f_{ij} : ij \in E(G) \}$
is an independent set in the algebraic matroid $\mathcal{A}(n)$. When this is the case, we will say \emph{$E(G)$ is an independent set in $\mathcal{A}(n)$}.  A graph $G$ that is simultaneously stress-free and rigid in dimension $n$,
is called \emph{isostatic}.  In the matroid language, this says that $E(G)$ is a basis
in the matroid $\mathcal{A}(n)$.

\begin{rmk}
Note that rigidity and being stress-free are properties that
hold generically.  There are situations where a graph is rigid but there exist non-generic
choices of the points $\bfp_{j}$ that make the resulting framework flexible. 
Since we are only interested in generic properties of the graph, we can ignore
such issues.
\end{rmk}

The second algebraic matroid we will study is the symmetric minor matroid $S(m,n)$.
In particular, let $ \cc[\Sigma] := \cc[\sigma_{ij}: 1 \leq i \leq j \leq m ]$
and let $I_{n+1}$ be the prime ideal of $(n+1)$-minors of the generic symmetric
matrix $\Sigma$. 
The symmetric minor matroid $S(m,n)$ is the algebraic matroid of the elements
$\sigma_{ij}: 1 \leq i \leq j \leq m$ in the extension $K( \cc[\Sigma]/I_{n+1})/\cc$.
In the language of algebraic matroids, the rank of the graph $G$ can be expressed
as follows.

\begin{prop}\label{prop:symminormatroid}
Let $G$ be a graph on vertex set $[m]$.  The rank of $G$ is the smallest $n$ such that 
$\{ \sigma_{ii}: i \in [m]\}  \cup \{ \sigma_{ij} : ij \in E(G) \}$ is
an independent set of the symmetric minor matroid $S(m,n)$.
\end{prop}

The ideal $I_{n+1}$ is the vanishing ideal of a parametrization, a
fact that we will use in connecting the matroid $S(m,n)$ to the matroid $\mathcal{A}(n-1)$.
Indeed, every symmetric matrix of rank $\leq n$ can be realized as $P^{T}P$ for some
$n \times m$ matrix $P$ (over $\cc$).  Hence, if we let 
$g_{ij}  =  \bfp_{i} \cdot \bfp_{j}$ for $1 \leq i \leq j \leq m$
then the algebraic matroid on these elements is the same as the algebraic
matroid $S(m,n)$.

For the rank problem of interest, that is the rank problem for studying the maximum likelihood threshold, we
are always looking at sets that contain all of the diagonal elements $\sigma_{ii}$.
Hence, we can look at independent sets in the matroid contraction by that collection
of elements.  From the standpoint of algebraic matroids, that amounts to studying the
algebraic matroid
$S(m,n)/{\rm diag}$, of the field extension $K( \cc[\Sigma]/I_{n+1})/\cc(\sigma_{ii}: i \in [m])$
with ground set consisting of the elements $\sigma_{ij}:  i < j$.

\begin{thm}\label{thm:matroidrelation}
The algebraic matroids $S(m,n)/ {\rm diag}$ and $\mathcal{A}(n-1)$ are isomorphic.
\end{thm}

To prove Theorem \ref{thm:matroidrelation} we use the fact that
the algebraic matroid over a field of characteristic zero is isomorphic to the representable
matroid obtained from evaluating the Jacobian at a generic point of the parameter
space.  We will make these evaluations for both of the matroids
$S(m,n)$ and $\mathcal{A}(n-1)$ and compare the results.  These particular
Jacobian matrices will appear at other points in the paper  so
we introduce them outside of the proof.

First, consider the map $f:  \rr^{n \times m}  \rightarrow \rr^{m(m-1)/2}$
with 
$$
f(P)  =   ( \| \bfp_{i} - \bfp_{j} \|_{2}^{2})_{1 \leq i < j \leq m},
$$
where we let $P = \begin{pmatrix} \bfp_{1} &  \cdots & \bfp_{m} \end{pmatrix}$.

The Jacobian $J(f, P)$ of this map is an $mn \times {m \choose 2}$ matrix.
The rows of $J(f, P)$ should be grouped into $m$ blocks of size $n$ corresponding to the
$m$ points $\bfp_{1}, \ldots, \bfp_{m}$. The $ij$ column of the Jacobian matrix is the
vector with zeros in all blocks except the $i$th and $j$th blocks which have
$\bfp_i- \bfp_j$ and $\bfp_j - \bfp_i$ respectively (we have ignored the extra factor
of $2$ that appears in all entries).  For example, for 
$m = 4$ the matrix $J(f, P)$ is
$$
\begin{pmatrix}
\bfp_1- \bfp_2 &   \bfp_1- \bfp_3 & \bfp_1- \bfp_4 & 0 & 0 & 0  \\
\bfp_2- \bfp_1 &  0 & 0 & \bfp_2- \bfp_3 & \bfp_2- \bfp_4 & 0 \\
0 & \bfp_3- \bfp_1 & 0  & \bfp_3- \bfp_2 & 0 & \bfp_3- \bfp_4  \\
0 & 0 & \bfp_4- \bfp_1 & 0 & \bfp_4- \bfp_2 & \bfp_4- \bfp_3
\end{pmatrix}.
$$

On the other hand, consider the parameterization map $g:   \rr^{n \times m}  \rightarrow \rr^{m(m+1)/2}$ 
$$
g(P)  =  (  \bfp_i \cdot \bfp_j)_{1 \leq i \leq j \leq m}.
$$
The Jacobian $J(g, P)$ of this maps is an $mn \times {m +1 \choose 2}$ matrix. 
The rows of $J(g,P)$ should be grouped into $m$ blocks of size $n$ corresponding to the
$m$ points $\bfp_{1}, \ldots, \bfp_{m}$.  When $i \neq j$, the $ij$
column in $J(g, P)$ has zero vectors in all blocks except for the $i$th and $j$th blocks, which
have $\bfp_j$ and $\bfp_i$, respectively.  When $i = j$, there is only one nonzero block,
which is $2 \bfp_i$.  For example, for 
$m = 4$ the matrix $J(g, P)$ is
$$
\begin{pmatrix}
2 \bfp_1 & \bfp_2 & \bfp_3 & \bfp_4  & 0 & 0 & 0 & 0 & 0 & 0  \\
0  & \bfp_1 & 0 &  0  & 2 \bfp_2 & \bfp_3 & \bfp_4 & 0 & 0 & 0  \\
0 &  0 & \bfp_1 & 0 & 0 & \bfp_2 & 0 & 2 \bfp_3 & \bfp_4 & 0 \\ 
0 &  0 & 0  & \bfp_1 & 0 & 0 &\bfp_2 & 0 & \bfp_3 & \bfp_4 
\end{pmatrix}.
$$

\begin{proof}[Proof of Theorem \ref{thm:matroidrelation}]
Note that since all polynomials involved are defined over the integers,
the underlying ground field can be changed to be $\cc, \rr,$ or $\qq$ without
changing the matroid in any of the matroids in question.

We also note that the rank of the Jacobian $J(g,P)$ does not change if 
we scale each point $\bfp_{i}$ by a nonzero constant $\lambda_{i}$. 
Indeed, performing such a scaling is equivalent to multiplying the 
rows corresponding the row block indexed by $i$ by $\lambda_{i}^{-1}$ and 
multiplying the column indexed by $ij$ by $\lambda_{i}\lambda_{j}$,
and row and column operations do not change the rank of a matrix.

Since the matrix $P$  is generic, we can assume that all of the %nonzero
coordinates
of each $\bfp_{i}$ are nonzero.
By choosing an appropriate scaling, we may assume that the $n$-th coordinate of each $\bfp_{i}$
 is equal to $1$. 
 We write this formally as $\bfp_{i} = {\bfp_{i}' \choose 1 }$, with $\bfp_{i}' \in \rr^{n-1}$ 
 which we can assume to be generic.  Let $P' = \begin{pmatrix} \bfp'_{1} &  \cdots & \bfp'_{m} \end{pmatrix} $.
 
 Divide the columns corresponding to pairs $ii$ by $2$. Then subtract the column corresponding 
 to pair $ii$ from each column corresponding to $ij$. Let $M$ be the resulting matrix.
The columns of $M$ corresponding to the tuples $ii$ are clearly linearly independent 
 of all other columns, because they are the only columns that contain nonzero 
 entries in the last position in each block.  Hence, when we contract by these 
 diagonal elements we can delete the last row from each block (since we get all zeros). 
 The resulting matrix, the
matrix that represents the matroid $S(m,n)/\diag$, is $J(f,P')$. 
Hence $S(m,n)/\diag$ is isomorphic to $\ca(n-1)$ as claimed.
\end{proof}

Theorem \ref{thm:matroidrelation} means that the rank of a graph can be precisely characterized in terms of the
independence condition in the matroid $\ca(n-1)$.

\begin{thm}\label{thm:equiv} Let $G=(V,E)$ be a graph.  Then $\rank(G) = n$ 
if and only if $E$ is an independent set in $\mathcal A(n-1)$ and is not an independent set in $\ca(n-2)$.
\end{thm}

Combined with Theorem \ref{thm:uhler}, which bounds the $\mlt(G)$ by $\rank(G)$, Theorem \ref{thm:equiv} implies Theorem \ref{thm:mainintro}, the main theorem stated in the Introduction.

%%%%%%%%%%%%%%%%%%%%%%%%%%%%%%%%%%%%%%%%%%%%%%%%
%%%%%%%%%%%%%%%%%%%%%%%%%%%%%%%%%%%%%%%%%%%%%%%%
%%%%%%%%%%%%%%%%%%%%%%%%%%%%%%%%%%%%%%%%%%%%%%%%
%%%%%%%%%%%%%%%%%%%%%%%%%%%%%%%%%%%%%%%%%%%%%%%%

\section{Basic Results on $\rank(G)$} \label{sec:basic}

In this section we catalogue some basic results about the
rank of a graph $G$ that follow immediately from the connection to combinatorial rigidity, 
including bounds on the number of edges that can be involved and graph constructions that preserve rank.

\begin{prop}\cite[Lemma 2.5.5]{Graver1993}  \label{prop:subgraph}
Let $G'=(V', E')$ and $G=(V,E)$ such that $G'$ is a subgraph of $G$.  Then if $E$ is independent in $\mathcal A(n-1)$, the set $E'$ is independent in $\mathcal A(n-1)$.  Consequently, $\rank(G') \leq \rank (G)$.
\end{prop}

%\begin{proof}
%The image $\phi_{G'}(\dim Sym(m,n))$
%is obtained from $\phi_{G}(\dim Sym(m,n))$ by eliminating the coordinates corresponding to edges and vertices that do not appear in $G'$.
%Note that $\dim \phi_{G}(\dim Sym(m,n)) = \#V + \#E$, if and only if $\phi_{G}$ is dominant, i.e.~generically
%onto.  This property is preserved after taking coordinate projections.
%\end{proof}

%Putting these results together, we deduce the following necessary condition
%for $G$ to have rank $n$.

The condition in the next theorem is called \emph{Laman's condition} 
in the combinatorial rigidity literature and is a necessary condition 
for a set to be independent in the rigidity matroid $\mathcal {A}(n-1)$.  
We state the theorem in terms of the rank of $G$, 
using the equivalence established in Theorem \ref{thm:equiv}.

\begin{thm}\label{thm:dimbound}\cite[Theorem 2.5.4]{Graver1993}
Let $G = (V,E)$ be a graph, and suppose that ${\rm rank}(G) \leq n$.  Then, for all
subgraphs $G' = (V', E')$ of $G$ such that $\#V' \geq n-1$ we must have
\begin{equation}  \label{eq:dimbound}
 \#E'  \leq  \#V'(n-1)  -  {n  \choose 2}.
\end{equation}
\end{thm}

 Laman's Theorem \cite{Laman1970} states that the condition 
 of Theorem \ref{thm:dimbound} is both necessary and sufficient 
 for a set to be independent in $\mathcal {A}(2)$, which combined
  with Theorem \ref{thm:matroidrelation} and Theorem \ref{thm:uhler} 
  gives us the following corollary in regards to the maximum likelihood threshold.

\begin{cor}\label{cor:lamansthm} Let $G = (V,E)$ be a graph, if for all subgraphs $G' = (V', E')$ of $G$
\begin{equation*}
 \#E'  \leq  2  (\#V')  -  3,
\end{equation*} 
then $\mlt(G) \leq 3$.
\end{cor}

Naively checking the conditions of Corollary \ref{cor:lamansthm} is ineffcient.
Howeevr, there is a polynomial times algorithm that can check if an edge set is independent in $\mathcal{A}(2)$ \cite{Jacobs1997}.

\begin{ex} Let $G$ be the complete bipartite graph $K_{3,3}$.  
The treewidth of $G$ is $3$, thus, by Buhl's bound in Proposition \ref{prop:Buhl}, 
we have $\mlt(G) \leq 4$.  Using Corollary \ref{cor:lamansthm}, we can obtain the improved bound
$\mlt(K_{3,3}) \leq 3$.  Since $K_{3,3}$ is not a forest, we deduce that $\mlt(K_{3,3}) = 3$. 
\end{ex}

In rigidity theory, there are many operations that take
an independent set and produce a new independent set on a larger number of vertices.
We review two of these here, vertex addition and edge splitting.  
We begin with \emph{vertex addition}, also called \emph{0-extensions}, 
and elaborate on some of the implications with respect to the maximum likelihood threshold.  
Again, we state the theorem in terms of the rank of $G$, using the equivalence established 
in Theorem \ref{thm:equiv}.  A variation of Proposition \ref{prop:vertexadd} with rank   replaced by Ben-David's Gaussian rank is proved independently in \cite{BenDavid2014}.

\begin{prop}[Vertex Addition]\label{prop:vertexadd}\cite[Lemma 11.1.1]{Whiteley1996}
Let $G=(V,E)$ be a graph such that ${\rm rank}(G)  \leq n$ and $\#V=m$.  Let $G'$
be a new graph obtained from $G$ by adding the vertex $v'$ and at most
$n-1$ edges connecting $v'$ to other vertices in $G$.  Then
${\rm rank}(G')  \leq n$, and, in particular, $\mlt(G') \leq n$.
\end{prop}

%\begin{proof}
%Let $G$ be a graph and fix $n$. Recall that the Jacobian $J(G)$ has full rank at a generic point of the parameter space if and only if $\rank(G) \leq n$.  Let $G'$ be constructed as in the proposition and let $i_1, \ldots, i_r$ be the vertices adjacent to $m+1$ in $G'$. The Jacobian $J(G')$ has the following block form:
%$$\left(\begin{array}{ccccc}J(G) & 0 & 0 & 0 & 0 \\0 & 2{\bf p_{m+1}} & {\bf p}_{i_1} & \cdots & {\bf p}_{i_r}\end{array}\right)$$
%Since each ${\bf p}_i$ is generic, the last $r+1$ columns are linearly independent as long as  $r \leq n-1$.  Therefore, since $J(G)$ has full rank at generic ${\bf p}$, $J(G')$ does as well.
%\end{proof}

\begin{defn}
Let $G$ be a graph and fix an integer $r$.  The $r$-\emph{core} of $G$ denoted
 $r\textrm{--}{\rm core}(G)$ is the graph obtained from $G$ by successively deleting vertices 
 of degree $< r$.  A graph is said to have \emph{empty $r$-core}, if  $r\textrm{--}{\rm core}(G)$
 has no vertices.
\end{defn}

Using  Proposition \ref{prop:vertexadd} inductively, we immediately deduce:

\begin{thm}\label{thm:core}  
Let $G$ be a graph with empty $n$-core.  Then ${\rm mlt}(G) \leq n$.
\end{thm}

While not as powerful as the splitting result from the next section, 
Theorem \ref{thm:core} already implies a number of nice consequences in
some simple cases.

\begin{cor}\label{cor:grid}
Let $Gr_{k_{1},k_{2}}$ denote the $k_{1} \times k_{2}$ grid graph with $k_{1}, k_{2} \geq 2$.  Then
$\mlt(Gr_{k_{1},k_{2}})  = 3$.
\end{cor}

\begin{proof}
First of all,  $\mlt(Gr_{k_{1},k_{2}}) \geq 3$, since $Gr_{k_{1},k_{2}}$ contains a cycle.
On the other hand, ${\rm rank}(Gr_{k_{1},k_{2}})  \leq 3$, since $Gr_{k_{1},k_{2}}$
has empty $3$-core.  This can be seen by removing the corner vertices, which successively
leaves a new vertex of degree $2$.  Hence, 
$$3 \leq \mlt(Gr_{k_{1},k_{2}}) \leq {\rm rank}(Gr_{k_{1},k_{2}}) \leq 3$$
completes the proof.  
\end{proof}

%Note, however, that even for planar graphs the bound obtained by 
%looking at $r$-cores is not optimal for every graph, and often ${\rm rank}(G)$
%will be smaller.  For instance, for the edge graph of the icosahedron $I_{12}$, a
%graph with $12$ vertices and such that every vertex has degree $5$, we see that
%this graph has empty $6$-core, but is itself a $5$-core.  
%The splitting arguments we will introduce in the next section implies that
%${\rm rank}(I_{12}) = 4$.

Another well-known graph operation preserving rank is \emph{edge-splitting}.

\begin{thm}[Edge splitting]\cite[Theorem 11.1.7] {Whiteley1996}  Let $G=(V,E)$ 
be a graph such that $\rank(G) \leq n$, and let $e=\{v_1, v_2\} \in E$. 
Let $G'$ be a graph obtained from $G$ by removing $e$ and then adding 
a new vertex $v'$ such that $v'$ is attached to the vertices 
$v_1$ and $v_2$ and at most $n-2$ other vertices in $V$.  Then ${\rm rank}(G') \leq n$, and, in particular $\mlt(G') \leq n$.
\end{thm}

\begin{ex} \label{ex:lattice}
Consider the lattice graph $L_{(2,4)}$ pictured in 
Figure \ref{fig:edgeSplittingEx1}. The graph $L_{(2,4)}$ has 
tree-width 4 and contains the complete graph on 4 vertices, 
therefore $4 \leq \mlt(L_{(2,4)}) \leq 5$.  Denote the graph 
pictured in Figure \ref{fig:edgeSplittingEx2} by $G$. The graph $G$ has 
an empty $4-core$, thus $\mlt(G) \leq 4$.  We can obtain $L_{(2,4)}$ by 
removing the edge $\{2,5\}$ and adding the vertex 8 and the 
edges $\{1,8\}$, $\{2, 8\}$, $\{5,8\}$, and $\{7,8\}$, thus 
$L_{(2,4)}$ can be obtained from $G$ through an edge splitting and we have $\mlt(L_{(2,4)})=4$.
\end{ex}

\begin{figure}
\centering
\begin{minipage}{.5\textwidth}
  \centering
  \includegraphics[width=.4\linewidth]{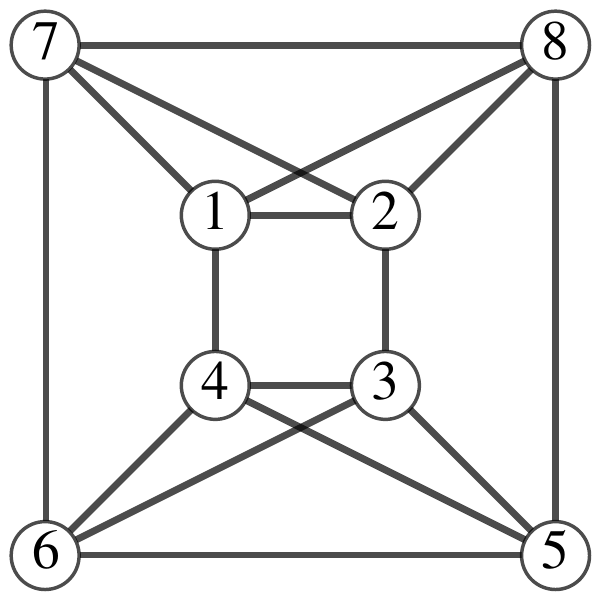}
  \captionof{figure}{The lattice graph $L_{(2,4)}$.}
  \label{fig:edgeSplittingEx1}
\end{minipage}%
\begin{minipage}{.5\textwidth}
  \centering
  \includegraphics[width=.4\linewidth]{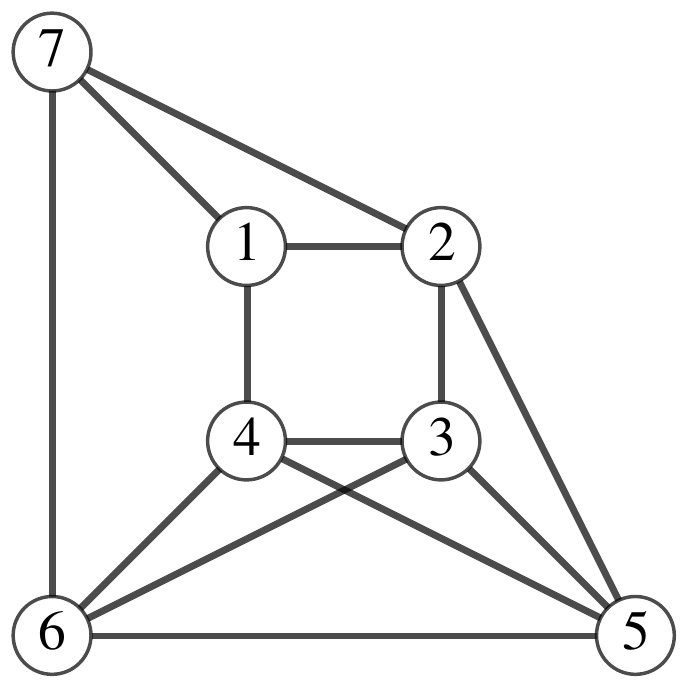}
  \captionof{figure}{Graph with $\mlt(G) \leq 4$}
  \label{fig:edgeSplittingEx2}
\end{minipage}
\end{figure}

Via a more advanced application of results of combinatorial rigidity theory, we can deduce the following
bound on the rank of any planar graph.

\begin{cor}
Let $G$ be a planar graph.  Then $\mlt(G) \leq 4$. 
\end{cor}

This proof is essentially due to Gluck \cite{Gluck1975} and depends on
 Dehn's \cite{Dehn1916} strengthening of Cauchy's theorem.

\begin{proof}
Every planar graph is a subgraph of a maximal planar graph, that is a planar graph
where it is not possible to add any further edges and maintain planarity.
By Proposition \ref{prop:subgraph} it suffices to prove the bound for such
maximal subgraphs.  The theorem is clearly true if $\#V \leq 3$, so assume 
that $\#V \geq 4$.  

Now, every maximal planar graph with $\#V \geq 4$ is $3$-connected.
Indeed, if a graph were not $3$-connected, $2$ vertices could be removed from $G$ leaving a disconnected graph, 
then an edge could be added from one of the components to  another without disrupting 
the planarity property.  Thus, if a planar graph is not $3$-connected, it is not maximal.

Every $3$-connected planar graph is the edge graph of a simplicial convex 
 polytope via Steinitz Theorem (see \cite[Ch.~3]{Ziegler1995}).
Dehn's theorem  \cite{Dehn1916} implies that the framework of any $3$-dimensional simplicial convex 
polytope is infinitesimally rigid in three dimensions, and hence the associated graph is generically rigid
in three dimensions.
Since a maximal planar graph with $m$ vertices has exactly $3m - 6$ edges, Dehn's theorem  combined
with Laman's criterion implies that any maximal planar graph is isostatic in $3$ dimensions. 
Hence, the set of edges of a maximal planar graph is an independent set in $\ca(3)$, which
implies that $\rank(G) \leq 4$.
\end{proof}

%\begin{thm}[Gluing Lemma 11.1.9] Let $G_1=(V_1, E_1)$ and $G_2=(V_2, E_2)$ be graphs and $G=(V_1 \cup V_2, E_1 \cup E_2$.  If $\rank(G_1) \leq n$ and $\rank(G_2) \leq n$ and $E_1 \cap E_2$ is generically $n-1$ rigid, then $\rank(G) \leq n$.
%\end{thm}

%%%%%%%%%%%%%%%%%%%%%%%%%%%%%%%%%%%%%%%%%%%%%%%%
%%%%%%%%%%%%%%%%%%%%%%%%%%%%%%%%%%%%%%%%%%%%%%%%
%%%%%%%%%%%%%%%%%%%%%%%%%%%%%%%%%%%%%%%%%%%%%%%%
%%%%%%%%%%%%%%%%%%%%%%%%%%%%%%%%%%%%%%%%%%%%%%%%

\section{Splitting Theorem}\label{sec:splitting}

In this section we prove a theorem that allows us to relate the 
rank of a graph to the rank of smaller subgraphs, at the expense of
needing to calculate the birank of some associated bipartite graphs.
The birank is the bipartite analogue of rank of a graph, and is
naturally related to the theory of \emph{bipartite rigidity} introduced
in \cite{Kalai2013}.  This splitting theorem allows us to give a number
of simple computations of $\rank(G)$ and hence gives us a simple
way to compute bounds on $\mlt(G)$.

For a bipartite graph $G=(V_{1}, V_{2},E)$  with a fixed bipartition of the vertices
$V=V_{1} \sqcup V_{2}$ where $\# V_1=m_1$, $\# V_2=m_2$ and two integers 
$r_{1}, r_{2}$, define the following linear space for generic 
points $X \in \cc^{m_1 \times r_1}$ and $Y \in \cc^{r_2 \times m_2}$:
$$L_{r_1, r_2}^{(X, Y)} := \{ X \cdot A +B \cdot Y \ : \ A \in \cc^{r_1 \times m_2}, \ B \in \cc^{m_1 \times r_2} \}.$$
Let 
$$ \phi_E \ : \cc^{m_1 \times m_2} \to \cc^{E}, \quad \quad 
\phi_{E}(\Sigma) = (\sigma_{ij})_{ij \in E}$$
be the coordinate projection that extracts the entries
 corresponding to edges of $G$.  Notice that while $\phi_G$ from 
 \eqref{eq:phiG} extracts entries corresponding to the diagonal, $\phi_E$ does not.

\begin{defn} Let $G=(V_{1}, V_{2},E)$ be  bipartite graph. 
Define the \emph{bipartite rank} of $G$, denoted $\birank(G)$, 
to be the set of all pairs of integers $(r_1, r_2)$ 
such that $\phi_E ( L_{r_1,r_2} ^{(X,Y)})= \cc^{E}$ 
for generic $X \in \cc^{m_1 \times r_1}$ and $Y \in \cc^{r_2 \times m_2}$.
\end{defn}

The case where $r_{1} = r_{2} = r$, the linear space $L_{r, r}^{(X, Y)}$
is the tangent space of the set of $m_{1} \times m_{2}$ matrices of rank $r$ at
the point $XY$.  Hence, bipartite rank in this case tells us about
independent sets in the algebraic matroid of the ideal of $(r+1)$-minors of
a generic matrix.  This matroid was studied in the context of matrix
completion problems in \cite{ KRT2013, Kiraly2012, Singer2010}.

The following proposition describes a method for constructing a 
new bipartite graph $G'$ from $G$ such that if $(r_1, r_2) \in \birank(G)$ 
then $(r_1, r_2) \in \birank(G')$.

\begin{prop}\label{prop:birank} Let $G=(V_1, V_2,E)$ be a bipartite graph 
with fixed partition $V=V_1 \sqcup V_2 $ such that $\# V_1 = m_1$ and 
$\#V_2 =m_2$. Let $(r_1, r_2) \in$ \emph{birank}$(G)$ and let $G'$ be 
a new bipartite graph obtained from $G$ by adding the vertex $v'$
 to $V_1$ and at most $r_2$ edges connecting $v'$ to other vertices 
 in $V_2$.  Then $(r_1, r_2) \in \birank(G')$.
\end{prop}

This result is essentially Lemma 3.7 of \cite{Kalai2013}.

\begin{proof} Let $X' \in \cc^{(m_1+1) \times r_1}$ and $Y \in \cc^{r_2 \times m_2}$ be generic. Write 
$$X'=\left[\begin{array}{c}X \\x\end{array}\right]$$
where $X \in \cc^{m_1 \times r_1}$ and $x \in \cc^{r_1}$.  Since $X'$ is generic, $X$ and $x$ are both generic as well.

Let $E'$ be the edge set of $G'$. Let $w' \in \cc^{E'}$, which we will write as $w'=\left(\begin{array}{c}w \\u\end{array}\right)$ where $w \in \cc^{E}$ and $u \in \cc^{E' - E}$.  Since $(r_1, r_2) \in \birank(G)$, by the definition of bipartite rank, the image $\phi_E ( L_{r_1,r_2} ^{(X,Y)})= \cc^{ E}$, and thus there exists $A \in \cc^{r_1 \times m_2}, \ B \in \cc^{m_1 \times r_2}$ such that 
$$\phi_E(X \cdot A +B \cdot Y)=w.$$

Now, note that if $b \in (\cc^{r_2})^*$, then
$$
 \left[\begin{array}{c}X \\x\end{array}\right] \cdot A + \left[\begin{array}{c}B \\b\end{array}\right] \cdot Y\\=\left[\begin{array}{c}X \cdot A + B \cdot Y \\x \cdot A + b \cdot Y\end{array}\right]. 
$$
Thus, to show surjectivity of $\phi_{E'}$, we need to find a $b \in \cc^{r_2}$ such that 
\begin{equation} \label{eq:system}
\phi_{E'-E} (x \cdot A + b \cdot Y)= u.
\end{equation}
However, the equation \eqref{eq:system} results a linear system with generic coefficients and $r_2$ unknowns (the entries of $b$).  Therefore, a solution always exists when $\# (E' - E) \leq r_2$ if $r_2 \leq m_2$, or $\# (E' - E) \leq m_2$ if $r_2>m_2$.
\end{proof}

For a bipartite graph $G$ with fixed bipartition of the vertices 
$V_{1}, V_{2}$ and two integers $r_{1}, r_{2}$, let $core_{r_{1},r_{2}}(G)$ be the graph
obtained from $G$ by repeatedly removing vertices of $G$ whenever $j \in V_{1}$ 
has degree less than or equal to $r_{2}$ or $j \in V_{2}$ has degree less than or equal to $r_1$.
Note that the $core_{r_{1},r_{2}}(G)$ is uniquely determined, despite the
fact that we have choices in the order we choose to remove vertices.
A graph is said to have empty $(r_{1},r_{2})$-core if $core_{r_{1},r_{2}}(G)$
have no vertices.  Clearly if $G$ has empty $(r_{1}, r_{2})$-core, it will have
$(r_{1}, r_{2}) \in \birank(G)$, in analogy to the relationship between
 ordinary rank of a graph and core.

\begin{ex}
A bipartite graph $G$ has empty $(1,1)$-core if and only if $G$ has no cycles.
\end{ex}

The notion of bipartite rank can help us understand the rank of an arbitrary 
(not necessarily bipartite) graph $G$. For a graph $G=(V,E)$ 
and disjoint subsets $V_{1}, V_{2} \subseteq V$, let
$G(V_{1}, V_{2})$ be the bipartite graph consisting of all edges $ij \in E(G)$
such that $i \in V_{1}$ and $j \in V_{2}$.  Let $G_{V_{1}}$ denote the induced
subgraph of vertex set $V_{1}$.

\begin{thm}[Splitting Theorem]\label{thm:splitting}
Let $G$ be a graph, $r_{1}, \ldots, r_{k}$ integers and  
$V_{1}, \ldots, V_{k}$ a partition of the vertices of
$G$, such that
\begin{enumerate}
\item for all $i$, $\rank(G_{V_{i}}) \leq r_{i}$ and 
\item for all $i \neq j$,  $(r_i, r_j) \in \birank( G(V_{i}, V_{j}))$. 
\end{enumerate}
Then $\rank(G)  \leq r_{1} + \cdots  + r_{k}$, and, in particular, $\mlt(G) \leq  r_{1} + \cdots  + r_{k}$.
\end{thm}

\begin{proof} Let $m_i = \# V_i$ for all $1 \leq i \leq k$. 
Let $m= \sum_{i=1}^k m_i$ and $n= \sum_{i=1}^k r_i$.  Recall that $Sym(m,n)$ can be parameterized 
over $\cc$ as
\[ Sym(m,n)=\{ P^TP \ : \ P \in  \cc^{n \times m} \}. \]
Using this parameterization, we see that the tangent space of $Sym(m,n)$ at the point $X=P^TP$ is
\[ T_X(Sym(m,n)) = \{ P^TA+A^TP \ : A \in \cc^{n \times m}\}. \]
To show that $\dim \phi_G(Sym(m,n)) = \#V + \#E$, we will show that the differential of $\phi_G$ at X
$$(D\phi_G)_X: T_X(Sym(m,n)) \to \cc^{V+E}$$
$$(D\phi_G)_X ( \Sigma) = \phi_G(\Sigma)$$
is surjective for a particular $X \in Sym(m,n)$.  This will imply $(D\phi_G)_X$ is surjective for generic $X \in Sym(m,n)$, and consequently $\phi_G$ restricted to $Sym(m,n)$ is dominant. 

Let $X=P^TP$ where $P$ is a block diagonal matrix of the form
\[ 
P= \begin{pmatrix}
P_1 & 0 & \cdots  & 0 \\
0 & P_2 &  \cdots  & 0 \\
\vdots & \vdots & \ddots & \vdots \\
0 & 0 & \cdots  & P_k
\end{pmatrix}
\]
such that $P_i \in \cc^{r_i \times m_i}$ is generic for $1 \leq i \leq k$. 
For every $A \in \cc^{n \times m}$, write $A$ as the block matrix
\[ 
A=\begin{pmatrix}
A_{11} & A_{12} & \cdots & A_{1k} \\
A_{21} & A_{22} & \cdots &  A_{2k} \\
\vdots & \vdots  & \ddots &  \vdots  \\
A_{k1} & A_{k2} & \cdots   & A_{kk}
\end{pmatrix}
\]
where $A_{ij} \in \cc^{r_i \times m_j }$. Then $P^TA+A^TP \in T_X(Sym(m,n))$ is a symmetric block matrix where the $(i,j)$th block is $P_i^TA_{ij}+A_{ji}^TP_j$, i.e.,
\[ 
P^TA+A^TP = \begin{pmatrix}
P_1^TA_{11}+A_{11}^TP_1 & \cdots & P_1^TA_{1k}+A_{k1}^TP_k \\
\vdots  & \ddots &  \vdots   \\
P_k^TA_{k1}+A_{1k}^TP_1 &  \cdots  & P_k^TA_{kk}+A_{kk}^TP_k\end{pmatrix}. 
\]

To prove surjectivity of $(D\phi_G)_X$, let $w \in \cc^{V+E}$, which can be written in the block form,
$ w= (w_{11}, w_{12}, \ldots,  w_{kk})$
where $w_{ii} \in \cc^{V_i + E(G_i)}$ for all $1 \leq i \leq k$ and 
$w_{ij} \in \cc^{E(G(V_i, V_j))}$ for all $1 \leq i < j \leq k$. 
Since $\rank(G_{V_i}) \leq r_i$ and $P_i$ is generic, the image of the 
linear space $\{ P_i^TA_{ii} + A_{ii}^TP_i \ : A_{ii} \in \cc^{r_i \times m_i} \}$ 
under the map $\phi_{G_{V_i}}$ is $\cc^{V_i + E(G_i)}$, which means there exists a 
$A'_{ii} \in \cc^{r_i \times m_i}$ such that 
$\phi_{G_{V_i}} (P_i^TA'_{ii} + A_{ii}^{'T}P_i)=w_{ii}$. 
Furthermore, since $(r_i, r_j) \in \birank( G(V_i, V_j))$ for $i \neq j$, 
there exists $A'_{ij} \in \cc^{r_i \times m_j}$ and $A'_{ji} \in \cc^{r_j \times m_i}$ 
such that $\phi_{E(G(V_i, V_j))}(P_i^TA'_{ij} + A^{'T}_{ji}P_i)=w_{ij}$. 
Let $A' \in \cc^{n \times m}$ with $ij$-th block  $A'_{ij}$.  
Then $(D\phi_G)_X(P^TA' + A^{'T}P)=\phi_{G}(P^TA' + A^{'T}P)=w$, and we have shown 
surjectivity of the differential $(D\phi_G)_X$.
\end{proof}

Theorem \ref{thm:splitting} and repeated application of 
Proposition \ref{prop:birank} gives us the following corollary.

\begin{cor}\label{prop:emptycore}
Let $G$ be a graph, $r_{1}, \ldots, r_{k}$ integers and  
$V_{1}, \ldots, V_{k}$ a partition of the vertices of
$G$, such that
\begin{enumerate}
\item for all $i$, $\rank(G_{V_{i}}) \leq r_{i}$ and 
\item for all $i \neq j$,  $G(V_{i}, V_{j})$ has an empty $(r_i, r_j)$ core
\end{enumerate}
Then $\mlt(G)  \leq r_{1} + \cdots  + r_{k}$.
\end{cor}

The special case where all the $r_{i}$ are equal to one is easy to understand.

\begin{cor}\label{cor:colorcycles}
Let $G$ be a graph and $V_{1}, \ldots, V_{k}$ be a partition of the vertices
of $G$ such that
\begin{enumerate}
\item  for all $i$,  $V_{i}$ is an independent set of $G$ and
\item  for all $i \neq j$,  $G(V_{i}, V_{j})$ has no cycles.
\end{enumerate}
Then $\mlt(G)  \leq k$.
\end{cor}

Of course, a partition of the vertices of the graph into independent sets
is a proper coloring of the graph, so we seek proper graph colorings 
where the induced subgraph on pairs of colors has no cycles.  
Such a coloring is called an acyclic coloring of a graph, and
the smallest number of colors such that a graph has an acyclic coloring
with that many colors is the acyclic coloring number of the graph \cite{Grunbaum1973}.
We
conclude with some examples illustrating the use of the splitting theorem
and its corollaries.

\begin{ex}
Consider the lattice graph $L_{(2,4)}$ from Example \ref{ex:lattice}.
The partition of the vertices $V_{1} = \{1,5\}$, $V_{2} = \{2,6\}$,
$V_{3} = \{3,7\}$, and $V_{4} = \{4,8\}$
has each $V_{i}$ an independent set in $L_{(2,4)}$ and each
bipartite graph $L_{(2,4)}(V_{i}, V_{j})$ without cycles.  This implies
that $\rank(G) \leq 4$.
\end{ex}

\begin{ex}
Consider the octahedral graph $O_{6}$, pictured in Figure \ref{fig:octa}.
\begin{figure}[h] 
\includegraphics{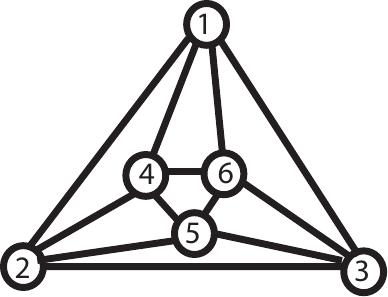}
\caption{\label{fig:octa} The octahedron graph $O_{6}$}
\end{figure}
The partition of vertices $V_{1} = \{1,4,5\}$ $V_{2} = \{2,3,6\}$ yields a splitting
that produces the bound $\rank(O_{6}) \leq 4$. Indeed, since $(O_{6})_{V_{1}}$ and $(O_{6})_{V_{2}}$ are
both trees, they have $\rank((O_{6})_{V_{1}}) = \rank((O_{6})_{V_{2}}) = 2$, and the bipartite graph
$O_{6}(V_{1}, V_{2})$ has empty $(2,2)$-core so $(2,2) \in \birank(O_{6}(V_{1}, V_{2}))$.
On the other hand $O_{6}$ has 12 edges, which by Theorem \ref{thm:dimbound} implies that $\rank(O_{6}) \geq 4$, so
the splitting proves that $\rank(O_{6}) = 4$.
\end{ex}

\begin{ex}
Consider the grid graphs $Gr_{k_{1},k_{2}}$.  Identify the vertices naturally
with $[k_{1}] \times [k_{2}]$.  Partition the vertices into three parts $V_{0}, V_{1}, V_{2}$
where
$$V_{i}  =  \{ (j_{1}, j_{2}) :  j_{1} + j_{2}  \equiv  i \mod 3 \}.$$
Clearly each $V_{i}$ is an independent set and each graph $G(V_{i}, V_{j})$ has no
cycles, so by Corollary \ref{cor:colorcycles}, $\rank(Gr_{k_{1}, k_{2}}) \leq 3$.
\end{ex}

The three preceding examples illustrating the Splitting Lemma can 
already be handled using the standard techniques
from rigidity theory from Section \ref{sec:basic}.
Let $d(G)$ be the maximal degree of the graph $G$ and $A(G)$ denote the
acylic coloring number.
In fact, Alon, McDiarmid and Reed \cite{Alon1991} showed that
$
A(G)  =  O( d(G)^{4/3})
$
and there exist graphs for which
$$
A(G)  =  \Omega\left( \frac{d(G)^{4/3}}{(\log d(G))^{1/3}} \right).
$$
On the other hand, based on the results from the previous sections,
$\rank(G) \leq d(G) + 1$, since any graph $G$ has
empty $(d(G) +1)$-core.  On the other hand, there are graphs
where $A(G) < d(G) + 1$.

\begin{ex}
Consider the graph $TGr_{k_{1}, k_{2}}$ the $k_{1} \times k_{2}$
torus grid graph.  This graph has $k_{1}k_{2}$ vertices, each of
degree $4$, and hence $2k_{1}k_{2}$ edges in total.
The core argument implies that $\rank(G) \leq 5$ whereas
from the edge count we see that $\rank(G) \geq 4$.

Suppose that $k_{1}$ is divisible by $4$  and $k_{2}$  is divisible by $3$.
Consider the $4 \times 3$ blocks of colors:
$$
B = \begin{pmatrix}
1 & 2 & 3 \\
2 & 3 & 4  \\
3 & 4 & 1  \\
4 & 1 & 2  
\end{pmatrix} 
$$
and consider the resulting coloring of $TGr_{k_{1}, k_{2}}$ 
obtained by repeating this block.
This coloring shows that $A( TGr_{k_{1}, k_{2}}) \leq 4$ since
each of the induced colorings on coloring classes $(i, i+1) \mod 4$
will consist of paths descending from the northeast to the southwest, 
that do not cross left-right boundaries 
from one $B$ to the next $B$.  Coloring classes $(i, i+2) \mod 4$
only involve edges that cross between adjacent left-right blocks,
so also do not produce cycles.
\end{ex}

%%%%%%%%%%%%%%%%%%%%%%%%%%%%%%%%%%%%%%%%%%%%%%%%
%%%%%%%%%%%%%%%%%%%%%%%%%%%%%%%%%%%%%%%%%%%%%%%%
%%%%%%%%%%%%%%%%%%%%%%%%%%%%%%%%%%%%%%%%%%%%%%%%
%%%%%%%%%%%%%%%%%%%%%%%%%%%%%%%%%%%%%%%%%%%%%%%%

\section{Weak Maximum Likelihood Threshold}\label{sec:weak}

A weaker notion of maximum likelihood threshold was also introduced in \cite{Buhl1993}
and further studied in \cite{Uhler2012}, which asks not for maximum likelihood
estimates to exist for almost all $\Sigma_{0} \in Sym(m,n) \cap \mathbb{S}^{m}_{\ge 0}$ but just for an open set of $Sym(m,n) \cap \mathbb{S}^{m}_{\ge 0}$.
This leads us to the notion of weak maximum likelihood threshold:

\begin{defn}
For each graph $G$, the \emph{weak maximum likelihood threshold},  $\wmlt(G)$, is the smallest 
$n$ such that there exists a
$\Sigma_{0}  \in Sym(m,n) \cap \mathbb{S}^{m}_{\ge 0}$
and a 
 $\Sigma \in \mathbb{S}^{m}_{>0}$ such that
$\phi_{G}(\Sigma_{0})  =  \phi_{G}(\Sigma)$.
\end{defn}

Note that because the positive definite cone $\mathbb{S}^{m}_{>0}$ is open,
the existence of a single matrix $\Sigma_{0}$ with this property guarantees an
open set of such matrices of positive measure in $Sym(m,n) \cap \mathbb{S}^{m}_{\ge 0}$.
Hence, we could also say that if $\wmlt(G) \geq n$, then maximum likelihood
estimates for the Gaussian graphical model associated to $G$ exist with
positive probability with $n$ data points.  Evaluating this probability would depend 
on having a specific distribution to draw the data from, for example
Buhl \cite{Buhl1993} calculated this for data drawn from an $\mathcal{N}(0, I_{m})$
distribution for the cycle graph.

 Clearly we have $\wmlt(G) \leq \mlt(G)$.  The two numbers can be equal, but often they are different.
Analogous to the splitting theorem for $\rank(G)$, there is also a straightforward splitting lemma
for $\wmlt(G)$.

\begin{lemma}[Splitting Lemma]\label{thm:splitting-wmlt}
Let $G$ be a graph, $r_{1}, \ldots, r_{k}$ integers and  
$V_{1}, \ldots, V_{k}$ a partition of the vertices of
$G$, such that for all $i$, $\wmlt(G_{V_{i}}) \leq r_{i}$.
Then $\wmlt(G)  \leq r_{1} + \cdots  + r_{k}$.
\end{lemma}

\begin{proof}
For $i = 1, \ldots, k$, let $\Sigma_{0}^{i} \in Sym(m,r_{i})$ and
$\Sigma^{i} \in \mathbb{S}^{\#V_{i}}_{>}$ such that 
$$\phi_{G_{V_{i}}}(\Sigma_{0}^{i}) = \phi_{G_{V_{i}}}(\Sigma^{i}).$$
Then the block diagonal matrices
$$
\Sigma_{0}  =  {\rm diag}(\Sigma_{0}^{1}, \ldots, \Sigma_{0}^{k})  \quad \mbox{ and } 
\Sigma  =  {\rm diag}(\Sigma^{1}, \ldots, \Sigma^{k})
$$
satisfy $\phi_{G}(\Sigma_{0})  =  \phi_{G}(\Sigma)$,   $\Sigma \in \mathbb{S}^{m}_{>0}$, and
$\Sigma_{0}  \in Sym(m,r_{1}+ \cdots + r_{k}) \cap \mathbb{S}^{m}_{\ge 0}$.
\end{proof}

The special case where all $r_{i} = 1$ yields the following corollary where
$\chi(G)$ denotes the chromatic number of $G$.

\begin{cor}\label{cor:wmlt-chrom}
Let $G$ be a graph.  Then $\wmlt(G) \leq \chi(G)$.
\end{cor}

So for example, every bipartite graph $G$ that has an edge satisfies $\wmlt(G) = 2$.
On the other hand, for the grid graphs $\mlt(Gr_{k_{1},k_{2}}) = 3$ so $\wmlt(G)$ 
is typically smaller that $\mlt(G)$.

At this point we know very little about the weak maximum likelihood threshold,
even for the graphs with $\wmlt(G) =2$.  Buhl showed that
if $C_{k}$ is a cycle of length $k \geq 4$, then $\wmlt(C_{k}) = 2$, while $\wmlt(C_{3}) = 3$.
  A corollary to this result
is the following necessary condition for a graph to have $\wmlt(G) = 2$.

\begin{cor}
Let $G = ([m],E)$ be a graph with $\wmlt(G) = 2$.  Then $G$ is triangle free and
there exists a cyclic order $w = w_{1}w_{2}\cdots w_{m}$ of the vertices of $G$ such that
for any subset $V \subset [m]$ such $G_{V}$ is a cycle, the induced cyclic ordering
$w_{V}$ is not a cycle ordering induced by the natural cyclic ordering from $G_{V}$.
\end{cor}

\begin{proof}
Let $\Sigma_{0} \in Sym(m,2) \cap \mathbb{S}^{m}_{\ge 0}$.
Then $\Sigma_{0}  = P^{T}P$ where $P = (\bfp_{1}, \ldots, \bfp_{m})$
and each $\bfp_{i} \in \rr^{2}$.  Scaling the $\bfp_{i}$ by nonzero constants $\lambda_{i}$
does not change whether or not there exists a $\Sigma$ (since we could
also scale the resulting $\Sigma$) so we can assume that all the $\bfp_{i}$ are in
the upper half-plane.  
Buhl showed that for the cycle graph $C_{k}$ with edges $(i,i+1)$, there exists
a $\phi_{C_{k}}((\Sigma_{0})  =  \phi_{C_{k}}((\Sigma))$ if and only if
the vectors $\bfp_{i}$ are not in cyclic order when considered by their angles in the upper 
half plane.

Let $\Sigma_{0}  \in Sym(m,n) \cap \mathbb{S}^{m}_{\ge 0}$
and   $\Sigma \in \mathbb{S}^{m}_{>0}$ such that
$\phi_{G}(\Sigma_{0})  =  \phi_{G}(\Sigma)$.
Then if $V$ is any subset of $[m]$ and $(\Sigma_{0})_{V}$ is the submatrix of $\Sigma_{0}$ obtained by deleting all rows and columns not indexed by vertices in $V$, then 
$(\Sigma_{0})_{V}  \in Sym(\#V,n) \cap \mathbb{S}^{\#V}_{\ge 0}$
and   $\Sigma_{V} \in \mathbb{S}^{\#V}_{>0}$ 
$\phi_{G_{V}}((\Sigma_{0})_{V})  =  \phi_{G_{V}}(\Sigma_{V})$.
Hence, taking $n = 2$, 
 by Buhl's result the vectors $\bfp_{i}$ must not appear in cyclic order for
 any cycle.  A necessary condition for finding such a set of vectors
 is the existence of a permutation with the prescribed property.
\end{proof}

If such an ordering $w$ of the vertices of a triangle free graph $G$ exists, then $G$ is said to satisfy \emph{Buhl's cycle condition}.  So a graph with $\wmlt(G) = 2$ satisfies Buhl's cycle
condition, but we do not know if the converse of this statement is true.  Also
we know of no example of a triangle-free graph that does not satisfy Buhl's cycle
condition.  Note that every triangle free graph $G$ with $\chi(G) \leq 3$ satisfies
Buhl's cycle condition, by choosing a $3$-coloring and listing the vertices in blocks
according to their color.

\begin{figure}[h]\label{fig:grotsch}
\includegraphics[width=150pt]{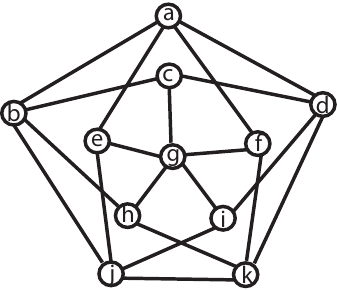} \caption{The  Gr\"otsch graph: the smallest
triangle free graph with $\chi(G) = 4$.  }
\end{figure} 
\begin{ex}
Consider the Gr\"otsch graph $G_{11}$, pictured in Figure 2, the
smallest triangle free graph with $\chi(G_{11}) = 4$.
This graph has a cyclic ordering of its vertices satisfying Buhl's cycle
condition, namely $achjbdefikg$.
On the other hand, the best upper bound on $\wmlt(G_{11})$
using Theorem \ref{thm:splitting-wmlt} comes from the
splitting $V_{1} = \{a,b,d,g,j,k\}$, $V_{2} = \{c,e,f,h,i\}$,
which yields $\wmlt(G_{11}) \leq 3$.  Is $\wmlt(G_{11}) = 2$?  
\end{ex}

%%%%%%%%%%%%%%%%%%%%%%%%%%%%%%%%%%%%%%%%%%%%%%%%
%%%%%%%%%%%%%%%%%%%%%%%%%%%%%%%%%%%%%%%%%%%%%%%%
%%%%%%%%%%%%%%%%%%%%%%%%%%%%%%%%%%%%%%%%%%%%%%%%
%%%%%%%%%%%%%%%%%%%%%%%%%%%%%%%%%%%%%%%%%%%%%%%%

\section{Score Matching Threshold}\label{sec:smt}

An alternative estimator to the maximum likelihood estimator for Gaussian graphical models is the \emph{score matching estimator} (SME) \cite{Hyvarinen2005}.  Unlike the MLE, the SME does not need to be computed iteratively, but instead is the solution to the set of linear equations.

The SME is an estimate of the concentration matrix $K = \Sigma^{-1}$. Let $G=(V,E)$ be a graph with $|V|=m$. Let $L_G$ be the following linear subspace of $\mathbb{S}_m$
 $$ L_G:= \{ K \in \mathbb{S}_m \ : \ K_{ij}=0 \text{ if } ij \notin E(G) \text{ and } i \neq j \},$$
and let $\Pi_G$ be the orthogonal projection from $\mathbb{S}_m$ onto $L_G$.  The estimating equations for the SME are

\begin{equation}\label{eq:estimatingequation}
\frac{1}{2}\cdot \Pi_G (K \Sigma_0^T + \Sigma_0 K^T)= I_m,
\end{equation}
where $\Sigma_0$ is the sample covariance matrix and $I_m$ is the $m \times m$ identity matrix.

Let $\Sigma_0=P^TP$ where $P=( \bfp_1, \ldots, \bfp_m)$ and each $\bfp_i \in \rr^{n}$.  In \cite{Forbes2015}, the authors give several equivalent conditions that are necessary and sufficient for the SME to exist, i.e. for the equation \eqref{eq:estimatingequation} to have a unique solution.  We will use the following:

\begin{prop}[\cite{Forbes2015}] Given a graph $G$, the SME exists if and only if $K=0$ is the only element of $L_G$ such that $K P^T=0$.
\end{prop}

While we do not know yet how large the difference between $\mlt(G)$ and $\rank(G)$ can be, we can show that given a graph $G$ the minimal observations $n$ needed to ensure that the SME exists almost surely is exactly equal to the rank of $G$.

\begin{defn} Let $G$ be a graph. The graph $G$ is \emph{$n$-estimable} if the score matching estimator of $\Sigma$ exists with probability one.  The \emph{score matching threshold of $G$}, denoted $\text{smt}(G)$ is the minimal $n$ such that $G$ is $n$-estimable.
\end{defn}

\begin{thm}\label{thm:smt} Let $G$ be a graph. Then
$$\smt(G)=\rank(G).$$
\end{thm}

\begin{proof}
The system $K P^T=0$ is a linear system in the entries of $K$ with coefficients in the entries of $P$.  The coefficient matrix $C$ of the system $K P^T=0$ is a $mn \times (\#V+\#E)$ matrix where the columns are indexed by the vertices and edges of $G$; the system $K P^T=0$ has a unique solution if and only if the rank of $C$ is $\#V + \# E$.

 Let $M$ be the matrix obtained from the Jacobian $J(g,P)$ from Section \ref{sec:crt} by scaling the columns indexed by $ii$ by $\frac{1}{2}$.  The coefficient matrix $C$ is the submatrix of $M$ obtained by selecting the columns indexed by the vertices and edges of $G$.  Thus, the matrix $C$ has rank $\# V+ \# E$ for generic $\bfp_1, \ldots, \bfp_m$ if and only if $\{\sigma_{ii} \ : i \in [m] \} \cup \{\sigma_{ij} \ : \ ij \in E(G)\}$ is an independent set of the symmetric minor matroid $S(m,n)$.  The statement then follows by Proposition \ref{prop:symminormatroid}.
\end{proof}

We can now apply all the results in the previous sections on the rank of a graph to the score matching threshold.  For example:

\begin{cor}\label{cor:smt3} Let $G = (V,E)$ be a graph. The $\smt(G) \leq 3$ if and only if for all subgraphs $G' = (V', E')$ of $G$
\begin{equation*}
 \#E'  \leq  2  (\#V')  -  3.
\end{equation*} 
\end{cor}

\begin{cor} Let $G$ be a graph with empty $n$-core.  Then $\smt(G) \leq n$.
\end{cor}

\begin{cor}\label{cor:colorcycles2}
Let $G$ be a graph and $V_{1}, \ldots, V_{k}$ be a partition of the vertices
of $G$ such that
\begin{enumerate}
\item  for all $i$,  $V_{i}$ is an independent set of $G$ and
\item  for all $i \neq j$,  $G(V_{i}, V_{j})$ has no cycles.
\end{enumerate}
Then $\smt(G)  \leq k$.
\end{cor}

Lauritzen stated the following conjecture about the
score matching threshold in his lecture at the 2014 Prague Stochastics
meeting.

\begin{conj}\label{conj:LF} The graph $G$ is $n$-estimable if and only if
$$\#V + \# E \leq nm  -  {n \choose 2}. $$
\end{conj}

The translation to rigidity the immediately provides counter examples.

\begin{ex}[Counterexample to Conjecture \ref{conj:LF}]  Let $G=(V,E)$ be the graph depicted in Figure \ref{fig:counterexample}.  Let $n=3$.  Then
$$\# V + \#E = 12 = nm  -  {n \choose 2}.$$
Thus, $G$ is conjectured to be $3$-estimable.  However, by Corollary \ref{cor:smt3}, this cannot be the case since the complete graph $K_4$ is a subgraph of $G$.
\end{ex}
\begin{figure}[h]
\includegraphics[width=100pt]{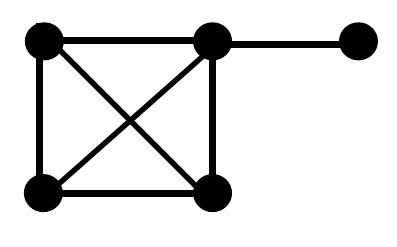} \caption{\label{fig:counterexample}Graph satisfies conditions of Conjecture \ref{conj:LF} for $n=3$, but is not $3$-estimable.  }
\end{figure}

Even with the stronger condition that the inequality
$\#V' + \# E' \leq n \#V'  -  {n \choose 2} $ for every 
induced subgraph $G' = (V', E')$ of $G$, there are known counterexamples of
graphs satisfying all of these inequalities but not being rigid.
The simplest such graph is the double banana graph, which satisfies
all these inequalities for $n = 4$ but is not a rigid graph in $\mathcal{A}(3)$.
So the Gaussian graphical model associated to the double banana graph
is not $4$-estimable.

%%%%%%%%%%%%%%%%%%%%%%%%%%%%%%%%%%%%%%%%%%%%%%%%
%%%%%%%%%%%%%%%%%%%%%%%%%%%%%%%%%%%%%%%%%%%%%%%%
%%%%%%%%%%%%%%%%%%%%%%%%%%%%%%%%%%%%%%%%%%%%%%%%
%%%%%%%%%%%%%%%%%%%%%%%%%%%%%%%%%%%%%%%%%%%%%%%%

\section{Conclusion}

The maximum likelihood threshold of a graph is an important
measure of the complexity of the Gaussian graphical model associated
to the graph.  It measures how much data is needed to calculate maximum
likelihood  estimates for the parameters of the model.
We showed that a result of Uhler implies that the maximum likelihood
threshold is closely related to combinatorial rigidity
theory, and then imported a number of results from
combinatorial rigidity theory to get new bounds on the maximum likelihood
threshold.  These new bounds significantly improve
bounds that exist in the literature, and in some cases imply 
effective ways to check for whether the MLE will exist almost surely.

We conclude here with two remaining questions. First, as discussed in the Introduction, does there exists a graph $G$ such that $\mlt(G)$ is strictly less than $\rank (G)$?  And secondly, is it possible to directly pin down the precise connection between rigidity theory and the maximum likelihood threshold? 
We provide a conjecture relating the maximum likelihood threshold to a stronger 
form of rigidity. Let $G=(V,E)$ be a graph with $\#V=m$. 
A \emph{framework} in $\rr^n$ with respect to $G$, denoted $(G,P)$, is an
$n \times m$ matrix $P$ such that the $i$th column of $P$, 
denoted $\bfp_i$, is an embedding of the $i$th vertex of $G$ into $\rr^n$.

\begin{defn} Two frameworks $(G, P)$ and $(G, Q)$ are \emph{edge-equivalent} if 
$$\| \bfp_{i} - \bfp_{j} \|_{2}^{2}=\| \bfq_{i} - \bfq_{j} \|_{2}^{2} \ \ \ \ \ \forall ij \in E(G). $$
\end{defn}

\begin{defn} Let $G=(V, E)$ be a graph with $\#V = m$.
 A framework $(G, P)$ in $\rr^n$ is 
\emph{$n$-dependently rigid} if for every edge equivalent 
framework $(G, Q)$ in $\rr^m$ the set of point $\{\bfq_{1}, \ldots, \bfq_{m}\}$ is
affinely dependent.
\end{defn}

We will say that $G$ is \emph{generically $n$-dependently rigid}  if every generic framework $(G, P)$ in $\rr^n$ is $n$-dependently rigid.

\begin{conj}\label{thm:MLEexistence} 
The maximum likelihood threshold for a graph $G$ is greater than $n$ if and only if $G$ is generically $n$-dependently rigid.
\end{conj}

We hope that once a precise connection between combinatorial rigidity theory and maximum likelihood estimation is established, new results on the $\mlt(G)$, guaranteed to be sharp, could be obtained.  

In addition to studying the maximum likelihood threshold, in this paper, we also looked at two related graph invariants, the weak maximum likelihood threshold and the score matching threshold. Little is understood about the weak maximum likelihood threshold, however, here we were able to show $\wmlt(G)$ is bounded above by the chromatic number of $G$, and we were able to give a necessary condition on $G$ for $\wmlt(G)=2$.  As for the score matching threshold, we showed a direct connection between the $\smt(G)$ and independent sets in the generic rigidity matroid.  While we saw that for $n=3$, conditions for independence in $\mathcal{A}(2)$ are efficient to check and some sufficient conditions for independence in $\mathcal{A}(n-1)$ are known for $n >3$, it should be noted that it is still an open problem to characterize all independent sets in $\mathcal{A}(3)$.  We hope that this connection though inspires more work on understanding the rigidity matroid for statistical applications.

\section*{Acknowledgments}

Elizabeth Gross was partially supported by the US National Science Foundation (DMS 1304167).
Seth Sullivant was partially supported by the David and Lucille Packard 
Foundation and the US National Science Foundation (DMS 0954865).  
We thank  Jan Draisma and Piotr Zwiernik
 for helpful discussions regarding the score matching estimator.

\end{document}